\documentclass{scrartcl}

\setkomafont{disposition}{\bfseries}

\usepackage{fullpage}
\usepackage{lmodern}
\usepackage[T1]{fontenc}
\usepackage{ae}
\usepackage{aecompl}

\usepackage{microtype} %[kerning, tracking, spacing]
\usepackage{mathtools}
\usepackage{amsthm, amssymb, commath}

\usepackage{tikz}
\usepackage{tkz-graph}
\usetikzlibrary{matrix}
 \usetikzlibrary{shapes}
\usetikzlibrary{arrows,decorations.markings}
\usepackage{tikz-cd}
\usepackage{enumerate}

\theoremstyle{definition}
\newtheorem{theorem}{Theorem} 
\newtheorem{proposition}[theorem]{Proposition}    
\newtheorem{definition}[theorem]{Definition}  
\newtheorem{prgrph}[theorem]{} 
\newtheorem{lemma}[theorem]{Lemma} 
 
\newtheorem{notation}[theorem]{Notation} 
 
\newtheorem*{ack}{Acknowledgements}

\newcommand{\Q}{\mathbb{Q}}
\newcommand{\C}{\mathbb{C}}
\newcommand{\M}{\overline{\mathcal{M}}}
\newcommand{\B}{\overline{\mathcal{B}}}
\renewcommand{\H}{\overline{\mathcal{H}}}
\newcommand{\Ba}{\B^{\textup{Adm}}}
\newcommand{\Ha}{\H^{\textup{Adm}}}

\DeclareMathOperator{\codim}{codim}

\DeclareMathOperator{\spec}{Spec}
\DeclareMathOperator{\supp}{Supp}
\newcommand{\Bo}{\mathcal{B}}
\newcommand{\Mo}{\mathcal{M}}
\newcommand{\Ho}{\mathcal{H}}

 \author{Jason van Zelm}
\title{Nontautological Bielliptic Cycles}
\date{}

\begin{document}

\maketitle

\begin{abstract}
Let $[\B_{2,0,20}]$ and $[\Bo_{2,0,20}]$ be the classes of the loci of stable resp. smooth bielliptic curves with~$20$ marked points where the bielliptic involution acts on the marked points as the permutation $(1\; 2)...(19\; 20)$. Graber and Pandharipande proved in \cite{Graber2001} that these classes are nontatoulogical. In this note we show that their result can be extended to prove that $[\B_{g}]$ is nontautological for $g\geq 12$ and that $[\Bo_{12}]$ is nontautological.
\end{abstract}

\section{Introduction}

 The system of \emph{tautological rings} $\{ R^\bullet(\M_{g,n})\}$ is defined to be the minimal system of $\Q$-subalgebras of the Chow rings $A^\bullet(\M_{g,n})$ closed under pushforward (and hence pullback) along the natural gluing and forgetful morphisms
\begin{align*}
 \M_{g_1,n_1+1}\times \M_{g_2,n_2+1}& \longrightarrow \M_{g_1+g_2,n_1+n_2},\\
 \M_{g,n+2}& \longrightarrow \M_{g+1,n},\\
 \M_{g,n+1}&\longrightarrow \M_{g,n}.
\end{align*}
The tautological ring $R^\bullet(\Mo_{g,n})$ of the moduli space of smooth curves is the image of~$R^\bullet(\M_{g,n})$ under the localization morphism $A^\bullet(\M_{g,n})\rightarrow A^\bullet(\Mo_{g,n})$. We will denote by $RH^{2\bullet}(\M_{g,n})$ the image of $R^\bullet(\M_{g,n})$ under the cycle map $A^\bullet(\M_{g,n})\rightarrow H^{2\bullet}(\M_{g,n})$ and define $RH^{2\bullet}(\Mo_{g,n})$ accordingly. We say a cohomology class is \emph{tautological} if it lies in the tautological subring of its cohomology ring, otherwise we say it is \emph{nontautological}. In this note we will work over $\C$ and all Chow and cohomology rings are assumed to be taken with rational coefficients.

These tautological rings are relatively well understood. An additive set of generators for the groups $R^\bullet(\M_{g,n})$ is given by decorated boundary strata and there exists an algorithm for computing the intersection product (see \cite{Graber2001}). The class of many ``geometrically defined'' loci can be shown to be tautological, for example this is the case for the class of the locus $\H_g$ of hyperelliptic curves in $\M_g$ (see \cite[Theorem 1]{Faber2005}).

Any odd cohomology class of $\M_{g,n}$ is nontautological by definition. Deligne proved that $H^{11}(\M_{1,11})\neq 0$, thus providing a first example of the existence of nontautological classes. In fact it is known that $H^\bullet(\M_{0,n})=RH^\bullet(\M_{0,n})$ (see \cite{Keel}) and that $H^{2\bullet}(\M_{1,n})=RH^{2\bullet}(\M_{1,n})$ for all $n$ (see \cite[Corollary 1.2]{petersen2014}).
 
Examples of geometrically defined loci which can be proven to be nontautological are still relatively scarce.  In \cite{Graber2001} Graber and Pandharipande hunt for algebraic classes in $H^{2\bullet}(\M_{g,n})$ and $H^{2\bullet}(\Mo_{g,n})$ which are nontautological. In particular they show that the classes of the loci~$\B_{2,0,20}$ and $\Bo_{2,0,20}$ of stable resp. smooth bielliptic curves of genus 2 with 20 marked points where the bielliptic involution acts on the set of marked points as the permutation $(1\; 2)...(19\; 20)$ are nontautological. They also show that for sufficiently high odd genus $h$ the class of the locus of stable curves of genus $2h$ admitting a map to a curve of genus $h$ is nontautological in $H^{2\bullet}(\M_{2h})$. Their result relies on the existence of odd cohomology in $H^\bullet(\M_{h,1})$ which has been proven to exist in \cite{pikaart1994} for all $h\geq 8069$. A recent survey of different methods of obtaining nontautological classes can be found in \cite{faber2013}.

In this note we prove the following two new results.
 \begin{theorem}\label{thmain}
 The cohomology class $[\B_{g,n,2m}]$ is nontautological for all $g+m\geq 12$, $0\leq n \leq 2g-2$ and $g\geq 2$.
\end{theorem}

\begin{theorem}\label{thmainop}
The cohomology class $[\Bo_{g,0,2m}]$ is nontautological when $g+m=12$ and $g\geq 2$.
\end{theorem}
 With Theorem \ref{thmain} we improve the genus for which algebraic nontautological classes on $\M_{g}$ are known to exists from 16138 to 12. As far as the author is aware, Theorem \ref{thmainop} provides the first example of a nontautological algebraic class on $\Mo_g$.

 \begin{ack}
 The author would like to thank his PhD supervisor Nicola Pagani for the many helpful discussions leading up to this paper. The author is supported by a GTA PhD fellowship at the University of Liverpool.
\end{ack}

 \section{Preliminaries}

 Admissible double covers were introduced to compactify moduli spaces of double covers of smooth curves, let us recall the definition:

\begin{definition}
 Let $(S,x_1,...,x_k,y_1,...,y_{2m})$ be a stable pointed curve of arithmetic genus $g$. An \emph{admissible double cover} is the data of a stable pointed curve $(T,x'_1,...,x'_{k},y'_1,...,y'_m)$ of arithmetic genus $g'$ and a 2-to-1 map $f\colon S\rightarrow T$ satisfying the following conditions:
 \begin{itemize}
  \item the restriction to the smooth locus $f^{\text{sm}}\colon S^{\text{sm}}\rightarrow T^{\text{sm}}$ is branched exactly at the points $x'_1,...,x'_k$ and the inverse image of $x'_i$ is $x_i$ for all $i=1,...,k$,
  \item the inverse image of $y'_i$ under $f$ is $\{y_{2i},y_{2i+1}\}$,
  \item the image under $f$ of each node is a node.
 \end{itemize}
We call $S$ the \emph{source} curve and $T$ the \emph{target} curve of the admissible cover. An \emph{admissible hyperelliptic structure} on $S$ is an admissible cover where $g'=0$ and an \emph{admissible bielliptic structure} on $S$ is an admissible cover with $g'=1$. Note that the admissible double cover $S\rightarrow T$ induces an involution on $S$ fixing the points $x_1,...,x_k$ and permuting the points $y_1,...,y_{2m}$ pairwise.
\end{definition}

\noindent One can define families of admissible double covers and isomorphisms between them (see \cite[Section 4]{acv}). By using the Riemann-Hurwitz formula and by induction on the number of nodes we can deduce that the number $k$ in the above definition equals $2g+2-4g'$. We denote the moduli stack of admissible bielliptic covers with $2m$ marked points switched by the involution by $\Ba_{g,2m}$. When $m=0$ we simply write $\Ba_g$. 

A natural target map and source map from each moduli space of admissible double covers can be defined as follows. The target map is a finite surjective map which sends each admissible cover to the target stable pointed curve  $(T,x'_1,...,x'_{k},y'_1,...,y'_m)\in \M_{g',k+m}$. From the properness of $\M_{g',k+m}$ we deduce that the space of such admissible covers is proper. The dimension of the space of such admissible double covers equals $2g-g'+2m-1$.    In the bielliptic case we get 
\begin{align*}
 \dim \Ba_{g,2m}&= 2g-2+m.
\end{align*}
The source map forgets all the structure of an admissible double cover except for 
\[(S,x_1,...,x_{k},y_1,...,y_{2m})\in \M_{g,k+2m}.\] 
In the bielliptic case this gives a  map $\Ba_{g,2m}\rightarrow \M_{g,2g-2+2m}$. We can compose this map with a composition of forgetful maps $\M_{g,2g-2+2m} \rightarrow \M_{g,n+2m}$ which forgets the first $2g-2-n$ points (which therefore correspond to the first $2g-2-n$ ramification points of the admissible bielliptic covers) and stabilizes. We denote by $\B_{g,n,2m}$ the image substack of $\Ba_{g,2m}$ in $\M_{g,n+2m}$. The above discussion can be summarized in the following diagram:
\begin{center}
\begin{tikzcd}
 \Ba_{g,2m}\arrow{d}\arrow{r} & \B_{g,n,2m} \arrow[hook]{r} &\M_{g,n+2m}\\
 \M_{1,2g-2+m}
\end{tikzcd}.
\end{center}
 The moduli stack $\Bo^{\text{Adm}}_{g,2m}$ is the open dense substack of $\Ba_{g,2m}$ of admissible bielliptic covers of smooth curves and we denote its image stack in $\Mo_{g,n+2m}$ by $\Bo_{g,n,2m}$. We have well defined Chow classes 
\begin{align*}
[\B_{g,n,2m}]\in  A^{g-1+n+m}(\M_{g,n+2m})\\
[\Bo_{g,n,2m}]\in  A^{g-1+n+m}(\Mo_{g,n+2m}).
\end{align*}
We will abuse notation and also denote the image of these classes in the respective cohomology rings by $[\B_{g,n,2m}]$ and $[\Bo_{g,n,2m}]$. In a completely analogous way, we can define spaces of admissible hyperelliptic covers $\Ha_{g,2m}$ and the loci $\H_{g,n,2m}$ and $\Ho_{g,n,2m}$ in $\M_{g,n+2m}$ and $\Mo_{g,n+2m}$ for all~$0\leq n \leq 2g+2$.

\begin{notation}
 We will denote by $\M_{g,n}^D$ (resp. $\Mo^D_{g,n}$) the moduli stack parameterizing trivial \'{e}tale double covers 
 \[
f\colon (C_1;y_{1,1},...,y_{n,1})\cup (C_2; y_{1,2},...,y_{n,2}) \rightarrow (C;y_{1},...,y_n) 
 \]
 mapping two isomorphic stable (resp. smooth) curves $(C_1;y_{1,1},...,y_{n,1})\simeq (C_2; y_{1,2},...,y_{n,2})$ to a curve $(C;y_{1},...,y_n) \simeq (C_1;y_{1,1},...,y_{n,1})$ such that $f^{-1}(y_i)=(y_{i,1},y_{i,2})$. 
\end{notation}

\noindent Our proof of Theorem \ref{thmain} relies on the following result for pullbacks along gluing morphisms.

\begin{proposition}[{\cite[Proposition 1]{Graber2001}}] \label{prop:Kunneth}
 Let $\xi\colon \M_{g_1,n_1+1}\times \M_{g_2,n_2+1} \rightarrow \M_{g_1+g_2,n_1+n_2}$ be the gluing morphism and let $\gamma \in RH^{\bullet}(\M_{g_1+g_2,n_1+n_2})$, then 
 \[
  \xi^*(\gamma)\in RH^{\bullet}(\M_{g_1,n_1+1})\otimes RH^{\bullet}(\M_{g_2,n_2+1}).
 \]
\end{proposition}

We say that a cycle $\lambda\in H^\bullet (\M_{g_1,n_1})\otimes H^\bullet(\M_{g_2,n_2})$ \emph{admits a tautological K\"{u}nneth decomposition} if $\lambda \in RH^{\bullet}(\M_{g_1,n_1})\otimes RH^{\bullet}(\M_{g_2,n_2})$.

 \section{Proof of Theorem \ref{thmain} and \ref{thmainop}}

We are now ready to prove Theorem \ref{thmain}. We start by proving the following weaker result.
\begin{proposition}\label{easyprop}
We have 
  \[
  [\B_{g,0,2m}]\not\in RH^\bullet(\M_{g,2m})
 \]
 for $g+m=12$ and $g\geq 2$.
\end{proposition}

\begin{proof}%\let\qed\relax
 Let $\iota_1\colon\Mo_{1,11}\times \Mo_{1,11} \rightarrow \M_{1,11}\times \M_{1,11}$ be the inclusion and  $\iota_2 \colon\M_{1,11}\times \M_{1,11}\rightarrow \M_{g,2m}$ the gluing morphism which glues the corresponding first $g-1$ points of the two factors and orders the remaining points by sending the $k$'th marked point of the first curve to $2k-1$ and the $k$'th marked point of the second curve to $2k$. Let $\iota$ be the composition $\iota_2 \circ \iota_1$ and let $\Delta$ resp. $\Delta_o$ be the diagonal of  $\M_{1,11}\times \M_{1,11}$ resp. $\Mo_{1,11}\times \Mo_{1,11}$ so that $\iota_1^*([\Delta])=[\Delta_o]$. In Lemma \ref{lem:pullback} we will prove that $\iota^*([\B_{g,0,2m}])= \alpha [\Delta_o]$ for some $\alpha\in \Q_{> 0}$.  Let $\partial(\M_{1,11}\times \M_{1,11}):=((\partial \M_{1,11})\times \M_{1,11}) \cup (\M_{1,11} \times (\partial \M_{1,11}))$. Since the sequence 
\begin{center}
 \begin{tikzcd}
 A^{10}(\partial (\M_{1,11}\times \M_{1,11}))  \arrow{r} & A^{11} (\M_{1,11}\times \M_{1,11}) \arrow{r}{\iota_1^*}&  A^{11}(\Mo_{1,11}\times \Mo_{1,11}) \arrow{r} & 0  
 \end{tikzcd}
\end{center}
 is exact there exists a class $B\in A^{10}(\partial (\M_{1,11}\times \M_{1,11}))$ such that $\iota_2^*([\B_{g,0,2m}]) = \alpha [\Delta] + B$.

 The class $B$ admits a tautological K\"{u}nneth decomposition by Lemma \ref{lem:CohM11}.\ref{point2}. Given a basis $\{ e_i\}_{i\in I}$ for $H^\bullet(\M_{1,11})$ with dual basis $\{\hat{e}_i\}_{i\in I}$ the cohomology class of the diagonal can be written as 
 \[
 [\Delta]=\sum_{i\in I} (-1)^{\deg e_i} e_i \otimes \hat{e_i}.
 \]
 In particular since $H^{11}(\M_{1,11})\neq 0$ the diagonal $[\Delta]$ does not admit a tautological K\"{u}nneth decomposition. Since the pullback of a tautological class along a (composition of) gluing morphisms admits a tautological K\"{u}nneth decomposition by Proposition \ref{prop:Kunneth}, this shows that~$[\B_{g,0,2m}]$ is nontautological.
\end{proof}

\begin{lemma}\label{lem:pullback}
 Consider the composition of gluing morphisms $\iota\colon \Mo_{1,11}\times \Mo_{1,11} \rightarrow \M_{g,2m}$ defined above. We have $\iota^*(\B_{g,2m})= \alpha [\Delta_o]$ for some $\alpha \in \Q_{>0}$.
\end{lemma}

\begin{proof}
 Consider the fiber diagram
\begin{equation*}
\begin{tikzcd}
  F  \arrow{r} \arrow{d} &\Ba_{g,2m}\arrow{d}{\phi}\\
 \Mo_{1,11}\times \Mo_{1,11}  \arrow{r}{\iota}& \M_{g,2m}
 \end{tikzcd}
\end{equation*}
We will describe the fiber product $F$, or rather the push forward of its class to $\Mo_{1,11}\times \Mo_{1,11}$.

Consider the moduli stack  $\Mo_{1,11}^D$, there is a closed embedding $\Mo_{1,11}^D\rightarrow \Mo_{1,11}\times \Mo_{1,11}$, $(C_1\cup C_2\rightarrow C) \mapsto (C_1,C_2)$ with image the diagonal $\Delta_o$. 
We define a map $\eta\colon\Mo_{1,11}^D \rightarrow \Ba_{g,2m}$ as follows: on the source curves~$\eta$ attaches rational bridges $R_i$ between the corresponding marked points $y_{i,1}$ of $C_1$ and~$y_{i,2}$ of $C_2$ for all $1\leq i \leq g-1$ and on the target curve it attaches a rational curve $R_i'$ with two marked points to the corresponding marked point~$y_i$ of~$C$. The trivial double cover $C_1\cup C_2\rightarrow C$ then induces an admissible double cover
\[
 \left(C_1\cup C_2 \cup \bigcup_{i=1}^{g-1}R_i\, ;\, y_{g,1},y_{g,2},...,y_{11,1},y_{11,2}\right) \longrightarrow \left(C \cup \bigcup_{i=1}^{g-1}R'_i\, ; \, y_g,...,y_{11}\right),
\]
branched at the marked points of each $R'_i$, which maps each pair of marked points $y_{i,1}$, $y_{i,2}$ of $C_1\cup C_2 \cup \bigcup_{i=1}^{g-1}R_i$  to the corresponding marked point $y_i$ of $C \cup \bigcup_{i=1}^{g-1}R'_i$.

By the universal property of fiber products we get a map $\Mo_{1,11}^D\rightarrow F$. We claim that the composition $\Mo_{1,11}^D\rightarrow F\rightarrow F^{\text{red}}$  is a finite\footnote{As in \cite[Definition 1.8]{Vistoli1989}.} surjective morphism. The map $F\rightarrow \Mo_{1,11} \times \Mo_{1,11}$ is proper since properness is stable under base extension, the map~$\Mo_{1,11}^D\rightarrow \Mo_{1,11}\times \Mo_{1,11}$ is proper because $\M_{1,11}^D\rightarrow \M_{1,11}\times \M_{1,11}$ is proper. It follows that~$\Mo_{1,11}^D\rightarrow F$ is proper. Since the map $\Mo_{1,11}^D\rightarrow \Mo_{1,11}\times \Mo_{1,11}$ is quasifinite so is $\Mo_{1,11}^D\rightarrow F$. Since $\Mo_{1,11}^D\rightarrow F^{\text{red}}$ is proper and quasifinite and $F^{\text{red}}$ is of finite type (and reduced) it remains to check that this map induces a surjection on closed points.

By definition an object of $F$ over $\spec \C$ consists of a curve $\tilde{C}:=(\tilde{C_1},\tilde{C}_2)\in \Mo_{1,11}\times \Mo_{1,11}( \C)$, an object $(S\rightarrow T)\in \Ba_{g,2m}( \C)$ and an isomorphism $\gamma\colon \iota(\tilde{C})\xrightarrow{\sim} \phi(S\rightarrow T)$. To prove the claim we will show that $(\tilde{C},(S\rightarrow T),\gamma)$ is isomorphic to an object in the image of $\Mo_{1,11}^D( \C)$.
Let~$f\colon\tilde{C}_1\cup \tilde{C}_2 \rightarrow \iota(\tilde{C})$ be the map of curves induced by $\iota$, set  $C:=\iota(\tilde{C})$, $C_1:=f(\tilde{C}_1)$ and~$C_2:=f(\tilde{C}_2)$, let $\tau$ be the involution on $C$ induced by the bielliptic involution of~$S\rightarrow T$ and let~$Q_i$ be the node of $C$ corresponding to the $i$'th marking of~$\tilde{C}_1$ and $\tilde{C}_2$.

Since $C_1$ and $C_2$ are smooth there are two possibilities for the action of $\tau$ on~$C$: Either it fixes $C_1$ and $C_2$ or it switches the whole of $C_1$ with the whole of $C_2$. Suppose $\tau$ fixes~$C_1$ and~$C_2$.  By construction the involution $\tau$ maps marked points lying on $C_1$ to marked points lying on $C_2$ so this is only possible if $C$ has no marked points at all. In this case $\tau$ must fix the different strands of $C$ at each~$Q_i$. If the inverse image of $Q_i$ in $S$ were to be a rational bridge $R_i$ then this rational bridge would have 2 marked ramification points which are not nodes, but this would imply that~$\tau$ switches the nodes on the rational bridge and therefore switches the strands of~$C$ at~$Q_i$. It follows that the inverse image of each $Q_i$ in $S$ is a single node $\hat{Q}_i$. Since $C_1$ and $C_2$ are smooth, $\tau$ induces an involution on the set of nodes $\{\hat{Q}_1,...,\hat{Q}_{11}\}$. We can thus find distinct  $\hat{Q}_i$, $\hat{Q}_j\neq \tau(\hat{Q}_i)$ such that~$S- \{\hat{Q}_i,\tau(\hat{Q}_i),\hat{Q}_j,\tau(\hat{Q}_j)\}$ is connected. But this means that there are at least two nodes~$P_i$ and $P_j$ of $T$ such that $T-\{P_i,P_j\}$ is connected. This would imply that the arithmetic genus of~$T$ is at least 2, which is a contradiction.

% The involution $\tau$ acts on the marked points of $C$ as the involution $(g_1\; g_2)...(11_1\; 11_2)$ (and thus it switches the marked points lying on $C_1$ with the marked points lying on~$C_2$) which is only possible if $C$ has no marked points at all. 

We can therefore assume $\tau$ maps $C_1$ to $C_2$. Let us first suppose that $\tau$ does not fix all nodes, so there exist some distinct $i$, $j$ such that $\tau(Q_i)=Q_j$. Let $P$ be the image of $\{ Q_i,Q_j\}$ under the bielliptic map. Like before we see that $T\backslash \{P\}$ is connected and it therefore has arithmetic genus 0 (since by assumption the arithmetic genus of $T$ is $1$). However the arithmetic genus of $C_1\backslash \{Q_i,Q_j\})$ is 1 and the bielliptic map restricts to an isomorphism $C_1\backslash \{Q_i,Q_j\}\rightarrow T\backslash \{P\}$, which is a contradiction.

 We have thus proven that $\tau$ switches the components $C_1$ and $C_2$ and fixes the nodes $Q_i$, which implies that $((\tilde{C}_1,\tilde{C}_2), (S\rightarrow T),\gamma)$ is isomorphic to an object in the image of $\M_{1,11}^D( \C)$. This proves that the map $\Mo_{1,11}^D\rightarrow F^{\text{red}}$ is surjective. 

 It follows that the pushforward of $\Delta_o$ to ${\Mo_{1,11}\times \Mo_{1,11}}$ equals the pushforward of $F$ to $\Mo_{1,11}\times\Mo_{1,11}$ up to a scalar. Since \[\codim_{\Mo_{1,11}\times \Mo_{1,11}}\Delta_o=11 = \codim_{\M_{g,2m}}\Ba_{g,2m}\] we see that $\iota^*[\B_{g,2m}]=\alpha [\Delta_o]$ for some $\alpha\in \mathbb{Q}_{>0}$ and~$g+m=12$. 
\end{proof}

 \begin{lemma}\label{lem:CohM11}
  \begin{enumerate}[i]
   \item Every algebraic class of codimension $11$ in $\M_{1,11} \times \M_{1,11}$ supported on $\partial (\M_{1,11}\times \M_{1,11})$ admits a tautological K\"{u}nneth decomposition.\label{point2}
  
   \item Every algebraic class on $\M_{1,11}\times \M_{1,11}$ of complex codimension less than 11 admits a tautological K\"{u}nneth decomposition.\label{point1}  
  \end{enumerate}
 \end{lemma}

\begin{proof}
This is a slightly weaker version of \cite[Lemma 3]{Graber2001}, the proof given there required that $RH^{2\bullet} (\M_{1,n})=H^{2\bullet}(\M_{1,n})$ and $H^k(\M_{1,n})=0$ for $n<11$, for which there was no reference at the time of \cite{Graber2001}. The first equation is \cite[Corollary 1.2]{petersen2014}. The second condition follows from Getzlers' computations for $n<11$ in \cite{Getzler1998}.
\end{proof}

\begin{prgrph}
We have now concluded the proof of Proposition \ref{easyprop}. To prove Theorem \ref{thmain} it remains to show that $[\B_{g,n,2m}]$ is nontautological for all $n$, $g$, $m$ with $g+m>12$.
\end{prgrph}

\begin{proof}[Proof of Theorem \ref{thmain}.]
 We will show in Lemma \ref{lem:add1} and \ref{lem:add2} that if $[\B_{g,n,2m}]$ is nontautological then so  are $[\B_{g,n+1,2m}]$ for $n\leq 2g-3$, and $[\B_{g,n,2m+2}]$. In Lemma \ref{lemg+1} we will show that if $[\B_{g,1,0}]$ is nontautological then so is $[\B_{g+1}]$. Using these statements inductively, with base case the statement of Proposition \ref{easyprop}, we conclude that $[\B_{g,n,2m}]$ is nontautological for all $g+m\geq 12$.
\end{proof}

\begin{lemma}\label{lem:add1}
 If $[\B_{g,n,2m}]$ is nontautological and $n\leq 2g-3$ then so is $[\B_{g,n+1,2m}]$.
\end{lemma}
\begin{proof}
Let $\pi\colon \M_{g,n+1+2m}\rightarrow{\M}_{g,n+2m}$ be the morphism which forgets the first point and stabilizes. Since $\pi(\B_{g,n+1,2m})=\B_{g,n,2m}$ and $\dim \B_{g,n+1,2m}=\dim \B_{g,n,2m}$ we have $\pi_*[\B_{g,n+1,2m}]=\alpha [\B_{g,n,2m}]$ for some $\alpha\in \Q_{>0}$. Because the push forward of a tautological class by the forgetful morphism is tautological, the result follows.
\end{proof}

\begin{lemma}\label{lem:add2}
 If $[\B_{g,n,2m}]$ is nontautological then so is $[\B_{g,n,2m+2}]$.
\end{lemma}
\begin{proof}
 Suppose $n< 2g-2$ then by the previous result $[\B_{g,n+1,2m}]$ is nontautological. Consider the gluing morphism 
 \[
 \sigma\colon   \M_{g,n+2m+1}\times\M_{0,3}\rightarrow \M_{g,n+2m+2}
 \]
 which glues the first points of both curves together, then $\sigma^{-1}(\B_{g,n,2m+2})=\B_{g,n+1,2m}$.  
 
 Since $\codim_{\M_{g,n+2m+2}}\B_{g,n,2m+2}=\codim_{\M_{g,n+2m+1}}\B_{g,n+1,2m}$ it follows that $\sigma^*[\B_{g,n,2m+2}]=\alpha [\B_{g,n+1,2m}]$ for some $\alpha\in \Q_{>0}$. Since $\sigma$ is a gluing morphism and the pullback of a tautological class along $\sigma$ admits tautological K\"{u}nneth decomposition $[\B_{g,n,2m+2}]$ is nontautological.
 
 If $n=2g-2$ we can first prove that $[\B_{g,n-1,2m+2}]$ is nontautological in the same way by pulling back through the map $\M_{g,n+2m}\times\M_{0,3}\rightarrow \M_{g,n+2m+1}$ and then use Lemma \ref{lem:add1}.
\end{proof}

\begin{lemma}\label{lemg+1}
 If $[\B_{g,1,0}]$ is nontautological then so is $[\B_{g+1}]$.
\end{lemma}

\begin{proof}
 Let  $\epsilon\colon \M_{g,1}\times \M_{1,1}\rightarrow \M_{g+1}$ be the gluing morphism. From the description of the boundary divisors of $\Ba_{g+1}$ (see \cite[Page 1275-1276]{Pagani2016}) it follows that there exists  $\alpha, \beta\in \Q_{> 0}$ such that
  \[
  \epsilon^*[\B_{g+1}] = \alpha [\B_{g,1,0}\times \M_{1,1}] +\beta [\H_{g-1,0,2}\times \M_{1,1}^D] \in H^\bullet(\M_{g,1}\times \M_{1,1}).
 \]
  
 The class $[\H_{g-1,0,2}\times \M_{1,1}^D] $ admits a tautological K\"{u}nneth decomposition (since the class of the hyperelliptic locus is tautological by \cite[Theorem 1]{Faber2005} and therefore so is its pushforward under a gluing morphism with a tautological class). The class  $[\B_{g,1}\times \M_{1,1}]$ does not admit a tautological K\"{u}nneth decomposition by assumption. It follows by Proposition \ref{prop:Kunneth} that $[\B_{g+1}]$ is nontautological.
\end{proof}

\begin{prgrph}
We will now prove a similar result for the open locus of $\M_{g,2m}$ where $g+m=12$.
\end{prgrph}

\begin{proof}[Proof of Theorem \ref{thmainop}]
The case where $g=2$ is treated in \cite[Section 3]{Graber2001}. We use a similar argument to prove the remaining cases. 
The proof runs by contradiction. Suppose $ [\Bo_{g,0,2m}] \in RH^\bullet (\Mo_{g,2m})$ then there is some collection of cycles $Z_i$ in $\M_{g,2m}$, of complex codimension 11 and supported on $\partial \M_{g,2m}$ such that $\sum  [Z_i] + [\B_{g,0,2m}]$ is a tautological class. Consider again the gluing morphism $\iota_2\colon\M_{1,11} \times \M_{1,11}\rightarrow \M_{g,2m}$ as above. By assumption the pullback of $\sum [Z_i] + [\B_{g,2m}]$ to $\M_{1,11} \times \M_{1,11}$ admits a tautological K\"{u}nneth decomposition whereby the pullback of $\sum [Z_i]$ to $\M_{1,11} \times \M_{1,11}$ must be nontautological.

   We shall use the usual notation that $\Delta_j$ is the locus of curves in $\M_{g,2m}$ consisting of two curves, one of which has genus $j$, glued together in a single node, and $\Delta_{\text{irr}}$ is the locus that generically parametrizes irreducible singular curves. Since $\iota_2(\M_{1,11} \times \M_{1,11})$ does not have a separating node we see that $\iota_2(\M_{1,11}\times \M_{1,11})\not\subset \Delta_j$. The intersection 
   \[
    \Delta_j \cap (\M_{1,11}\times \M_{1,11})
    \]
    therefore lies in $\partial (\M_{1,11}\times \M_{1,11})$. It follows by Lemma \ref{lem:CohM11}.\ref{point2} that $\iota_2^*[Z_i]$ admits a  tautological K\"{u}nneth decomposition if $\supp Z_i\subset \Delta_j$. 
 
 Consider now the $Z_i$ with support inside $\Delta_{\text{irr}}$. We can decompose the map $\iota_2$ as
 \begin{center}
  \begin{tikzcd}
   \M_{1,11}\times \M_{1,11}  \arrow{r}{\iota_2''} & \M_{g-1,2m+2} \arrow{r}{\iota_2'} &\M_{g,2m}
  \end{tikzcd}
 \end{center}
 Then there exist cycles $Y_i$ in $\M_{g-1,2m+2}$ such that $\iota'_{2*}[Y_i]=[Z_i]$. Now 
 \begin{align*}
  \iota_2^*[Z_i] &= \iota''^*_2\iota'^*_1[Z_i]\\
  &=\iota''^*_2 (c_1(N_{\M_{g-1,2m+2}}\M_{g,2m})\cap [Y_i]).
 \end{align*}
 We see that $\iota^*_2[Z_i]$ decomposes as a product of algebraic classes of codimension less than $11$, which admit tautological K\"{u}nneth decomposition by Lemma \ref{lem:CohM11}.\ref{point1}.

We conclude that all the $[Z_i]$ have tautological K\"{u}nneth decomposition when pulled back to $\M_{1,11}\times \M_{1,11}$. Therefore $\iota_2^*(\sum [Z_i] + [\B_{g,0,2m}])$ does not admit a tautological K\"{u}nneth decomposition. It follows by Proposition \ref{prop:Kunneth} that $[\Bo_{g,2m}]$ is nontautological.
\end{proof}

\bibliography{nontautBE}{}
\bibliographystyle{alpha}

J.~van~Zelm, \textsc{Department of Mathematical Sciences, University of Liverpool, Liverpool, L69 7ZL, United Kingdom}
\par\nopagebreak
  \textit{E-mail address}: \texttt{jasonvanzelm@outlook.com}

\end{document}